\documentclass[12pt]{article}
\usepackage{fullpage,amsthm,amssymb,amsmath}
\usepackage{enumerate,color}
\usepackage{tikz}

\newtheorem{theorem}{Theorem}[section]
\newtheorem{lemma}[theorem]{Lemma}

\newtheorem{cor}[theorem]{Corollary}

\newtheorem{example}[theorem]{Example}

\newcommand\needed[1]{
}

\begin{document}

\title{Word-representability of Toeplitz graphs}

\author{
Gi-Sang Cheon
\thanks{Applied Algebra and Optimization Research Center, Department of Mathematics, Sungkyunkwan University, Suwon, Republic of Korea.
\texttt{gscheon@skku.edu}
}
\and
Minki Kim
\thanks{Department of Mathematical Sciences, KAIST, Daejeon, Republic of Korea.
\texttt{kmk90@kaist.ac.kr}
}
\and
Jinha Kim
\thanks{Department of Mathematical Sciences, Seoul National University, Seoul, Republic of Korea.
\texttt{kjh1210@snu.ac.kr}
}
\and
Sergey Kitaev
\thanks{Department of Computer and Information Sciences, University of Strathclyde, Glasgow, United Kingdom.
\texttt{sergey.kitaev@cis.strath.ac.uk}
}
}

\date\today

\maketitle

\begin{abstract}
Distinct letters $x$ and $y$ alternate in a word $w$ if after deleting in $w$ all letters but the copies of $x$ and $y$ we either obtain a word of the form $xyxy\cdots$ (of even or odd length) or a word of the form $yxyx\cdots$ (of even or odd length). A graph $G=(V,E)$ is word-representable if there exists a word $w$ over the alphabet $V$ such that letters $x$ and $y$ alternate in $w$ if and only if $xy$ is an edge in $E$. 

In this paper we initiate the study of word-representable Toeplitz graphs, which are Riordan graphs of the Appell type. We prove that several general classes of Toeplitz graphs are word-representable, and we also provide a way to construct non-word-representable Toeplitz graphs.  Our work not only merges the theories of Riordan matrices and word-representable graphs via the notion of a Riordan graph, but also it provides the first systematic study of word-representability of graphs defined via patterns in adjacency matrices. Moreover, our paper introduces the notion of an infinite word-representable Riordan graph and gives several general examples of such graphs. It is the first time in the literature when the word-representability of infinite graphs is discussed. \\


\end{abstract}

\maketitle

\section{Introduction}
In this paper, we merge the theories of Riordan matrices and word-representable graphs via the notion of a Riordan graph introduced recently in~\cite{CJKM}. More precisely, we focus on the studies of Riordan graphs of the Appell type, which are known in the literature as Toeplitz graphs. We give various (general) conditions on (non-)word-representability of Toeplitz graphs leaving a complete classification in this research direction as a grand open question.

In this paper, for a word (or letter) $x$, the word $\underbrace{xx\cdots x}_{\tiny k\mbox{ times}}$ is 
denoted by $x^k$. We also let $[n]$ denote the set $\{1,2,\ldots,n\}$.

\subsection{Riordan matrices}
For any integral domain $\kappa$, we consider the ring of formal power series
$$\kappa[[z]]=\left\{f=\sum_{n=0}^\infty f_nz^n\ |\  f_n\in
\kappa\right\}.$$
If there exists a pair of generating functions $(g,f)\in \kappa[[z]]\times \kappa[[z]]$ with $f(0)=0$ such that $g\cdot f^j=\sum_{i\ge0}\ell_{i,j}z^i$ for each integer $j\ge0$, then the matrix  $L=[\ell_{ij}]_{i,j\ge0}$ is called a  {\it Riordan
matrix}, or {\em Riordan array}, generated by $g$ and $f$. Usually, we write
$L=(g,f)$. Since $f(0)=0$, every Riordan matrix $(g,f)$ is an infinite
lower triangular matrix. In particular, if a Riordan matrix is
invertible then it is a {\it proper} Riordan matrix. Note that $(g,f)$ is
invertible if and only if $g(0)\ne0$, $f(0)=0$ and
$f^\prime(0)\ne0$.
Some well known Riordan matrices are as follows.

$\bullet$ The {\em Pascal triangle} 
$$P=\left({1\over 1-z},{z\over
1-z}\right)=\left(\begin{array}{cccccc}
1&0&0&0&0&\cdots\\
1&1&0&0&0&\cdots\\
1&2&1&0&0&\cdots\\
1&3&3&1&0&\cdots\\
1&4&6&4&1&\cdots\\
\vdots&\vdots&\vdots&\vdots&\vdots&\ddots
\end{array} \right).$$

$\bullet$ The {\em Catalan triangle} 
$$C=\left({{1-\sqrt{1-4z}}\over
2z},{{1-\sqrt{1-4z}}\over 2}\right)=\left(\begin{array}{cccccc}
1&0&0&0&0&\cdots\\
1&1&0&0&0&\cdots\\
2&2&1&0&0&\cdots\\
5&5&3&1&0&\cdots\\
14&14&9&4&1&\cdots\\
\vdots&\vdots&\vdots&\vdots&\vdots&\ddots
\end{array} \right).$$

 $\bullet$ The {\em Fibonacci matrix} 
$$F=\left({1\over 1-z-z^2},z\right)=\left(\begin{array}{cccccc}
1&0&0&0&0&\cdots\\
1&1&0&0&0&\cdots\\
2&1&1&0&0&\cdots\\
3&2&1&1&0&\cdots\\
5&3&2&1&1&\cdots\\
\vdots&\vdots&\vdots&\vdots&\vdots&\ddots
\end{array} \right).$$

For a Riordan matrix $(g,f)$ over $\mathbb{Z}$, the $(0,1)$-matrix
$L=[\ell_{ij}]_{i,j\ge0}$ defined by
\begin{eqnarray*}
 \ell_{ij}\equiv[z^i]gf^j\;({\rm mod\;2})
\end{eqnarray*}
is called a {\it binary Riordan matrix}, and it is denoted by ${\cal
B}(g,f)$. The {\em leading principal matrix of order $n$} in ${\cal
B}(g,f)$ (resp., $(g,f)$) is denoted by ${\cal B}(g, f)_n$ (resp.,
$(g,f)_n$).

\subsection{Riordan graphs}
The notion of a Riordan graph was introduced in~\cite{CJKM}. A simple {\em labelled} graph $G$ with $n$ vertices is a {\em Riordan
graph} of order $n$ if the adjacency matrix of $G$ is an $n\times n$
symmetric $(0,1)$-matrix given by
\begin{eqnarray*}
{\cal A}(G)={\cal B}(zg,f)_n+{\cal B}(zg, f)_n^T
 \end{eqnarray*}
 for some Riordan matrix $(g,f)$ over $\mathbb{Z}$. We denote such $G$
  by $G_n(g,f)$, or simply by $G_n$ when the matrix $(g,f)$ is understood from the context, or it is not important.
A simple {\em unlabelled} graph is a {\em Riordan graph} if at least
one of its labelled copies is a Riordan graph. However, only labelled graphs are of interest to us in this paper.

So, for a Riordan graph $G$ on $n$ vertices defined by $(g,f)$, the $n\times n$ adjacency matrix ${\cal A}(G)$ satisfies the following:
 \begin{itemize}
\item its main diagonal entries are all 0, and
\item its lower triangular part below the main diagonal is the $(n-1)\times (n-1)$ binary Riordan matrix ${\cal B}(g,f)_{n-1}$.
\end{itemize}

Note that the {\em infinite} graph 
\begin{equation}\label{inf-def} G:=G(g,f)=\lim_{n\to\infty}G_n(g,f)\end{equation} 
is well defined, and we call it the {\em infinite Riordan graph} corresponding to the pair $(g,f)$. 

For example, the Riordan graph $G_6\left(\frac{1}{1-z},z\right)$ in Figure~\ref{complete} is the complete graph $K_6$, while $G\left(\frac{1}{1-z},z\right)$ is the infinite complete graph. 
Another example of a Riordan graph is the {\em Catalan graph of order} $n$, which is defined as $G_n\left({1-\sqrt{1-4z}\over2z},{1-\sqrt{1-4z}\over 2}\right)$; see Figure~\ref{catalan} for the case of $n=6$.
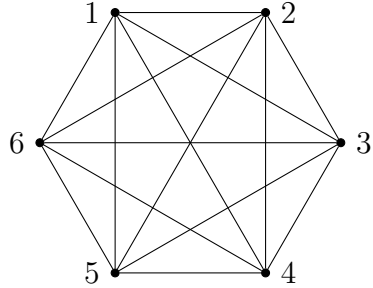
\begin{figure}[htbp]
\centering
\begin{tikzpicture}[main node/.style={fill,circle,draw,inner sep=0pt,minimum size=3pt}, scale=1]

\begin{scope}[xshift=-3cm, yshift=0cm]
\node (v) at (0,0){
 ${\cal A}\left(G_6\left(\frac{1}{1-z},z\right)\right)=\left(\begin{array}{cccccc}
        0 & 1 & 1 & 1 & 1 & 1 \\
        1 & 0 & 1 & 1 & 1 & 1 \\
        1 & 1 & 0 & 1 & 1 & 1 \\
        1 & 1 & 1 & 0 & 1 & 1 \\
        1 & 1 & 1 & 1 & 0 & 1 \\
        1 & 1 & 1 & 1 & 1 & 0
      \end{array}\right)$};
      \end{scope}

\begin{scope}[xshift=5cm, yshift=0cm]
\node[main node, label=left:1] (v1) at (120:2){};
\node[main node, label=right:2] (v2) at (60:2){};
\node[main node, label=right:3] (v3) at (0:2){};
\node[main node, label=right:4] (v4) at (300:2){};
\node[main node, label=left:5] (v5) at (240:2){};
\node[main node, label=left:6] (v6) at (180:2){};

\draw (v1)--(v2)--(v3)--(v4)--(v5)--(v6)--(v1);
\draw (v1)--(v3)--(v5)--(v1); \draw (v2)--(v4)--(v6)--(v2);
\draw (v1)--(v4); \draw (v2)--(v5); \draw (v3)--(v6);
\end{scope}
\end{tikzpicture}
\caption{$G_6\left(\frac{1}{1-z},z\right)$ is the complete graph $K_6$}
\label{complete}
\end{figure}  

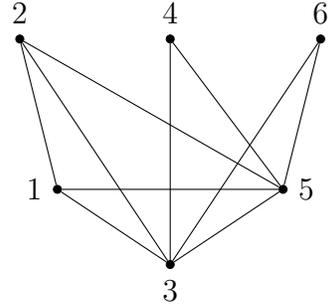
\begin{figure}[htbp]
\centering
\begin{tikzpicture}[main node/.style={fill,circle,draw,inner sep=0pt,minimum size=3pt}, scale=1]

\begin{scope}[xshift=-3cm, yshift=0cm]
\node (v) at (0,0){
 ${\cal A}\left(G_6\left({1-\sqrt{1-4z}\over2z},{1-\sqrt{1-4z}\over 2}\right)\right)=\left(\begin{array}{cccccc}
        0 & 1 & 1 & 0 & 1 & 0 \\
        1 & 0 & 1 & 0 & 1 & 0 \\
        1 & 1 & 0 & 1 & 1 & 1 \\
        0 & 0 & 1 & 0 & 1 & 0 \\
        1 & 1 & 1 & 1 & 0 & 1 \\
        0 & 0 & 1 & 0 & 1 & 0
      \end{array}\right)$};
      \end{scope}

\begin{scope}[xshift=5cm, yshift=-1.5cm]
\node[main node, label=left:$1$] (v1) at (-1.5,1){};
\node[main node, label=$2$] (v2) at (-2,3){};
\node[main node, label=below:$3$] (v3) at (0,0){};
\node[main node, label=$4$] (v4) at (0,3){};
\node[main node, label=right:$5$] (v5) at (1.5,1){};
\node[main node, label=$6$] (v6) at (2,3){};

\draw (v1)--(v2)--(v3)--(v4)--(v5)--(v6); \draw (v1)--(v3)--(v5)--(v1);
\draw (v2)--(v5); \draw (v3)--(v6);
\end{scope}
\end{tikzpicture}
\caption{$G_6\left({1-\sqrt{1-4z}\over2z},{1-\sqrt{1-4z}\over 2}\right)$, the Catalan graph of order $6$}
\label{catalan}
\end{figure}

\subsection{Toeplitz graphs}
A Riordan graph $G_n(g,f)$ with $f=z$ is called a {\em Riordan graph of the Appell type}. For example, the {\em Fibonacci graph} $G_n\left(\frac{1}{1-z-z^2},z\right)$ is of such a type. The class of Riordan graphs of the Appell type is also known as the class of {\em Toeplitz graphs}. Originally, a Toeplitz graph $G = (V, E)$ is defined as a graph with $V =[n]$ and $$E=\{ij\ |\ |i-j|\in\{t_1,\ldots,t_k\},1\leq t_1<\cdots< t_k\leq n-1\}.$$ See~\cite{Ghorban} and references therein for examples of results in the literature on Toeplitz graphs.

Throughout this paper, we denote by $G_n(a_1 a_2 \cdots a_m)$ the Toeplitz graph on $n$ vertices which is defined by $$G_n\left(\frac{b_1 + b_2 z + \cdots + b_m z^{m-1}}{1-z^m}, z\right),$$
where 
$a_i \equiv b_i$ $(\text{mod}~2)$.
For instance, the Fibonacci graph $G_n\left(\frac{1}{1-z-z^2},z\right)$ can be written as $G_n(110)$, or $G_n(1^2 0)$ since $\frac{1}{1-z-z^2}\equiv\frac{1+z}{1-z^3}\  (\text{mod}~2)$.

\subsection{Word-representable graphs}
Suppose that $w$ is a word over some alphabet and $x$ and $y$ are two distinct letters in $w$. We say that $x$ and $y$ {\em alternate} in $w$ if after deleting in $w$ all letters but the copies of $x$ and $y$ we either obtain a word of the form $xyxy\cdots$ (of even or odd length) or a word of the form $yxyx\cdots$ (of even or odd length).

A graph $G=(V,E)$ is {\em word-representable} if there exists a word $w$ over the alphabet $V$ such that letters $x$ and $y$ alternate in $w$ if and only if $xy$ is an edge in $E$. Such a word $w$ is called $G$'s {\em word-representant}. In this paper we assume $V$ to be $[n]=\{1,2,\ldots,n\}$ for some $n$. For example, the cycle graph on 4 vertices labeled by 1, 2, 3 and 4 in clockwise direction can be represented by the word 14213243. Note that a complete graph $K_n$ can be represented by any  permutation of $[n]$, while an edgeless graph (i.e.\ empty graph) on $n$ vertices can be represented by $1122\cdots nn$.

There is a long line of research on word-representable graphs, which is summarised in the recently published book \cite{KL} and the survey paper \cite{K}. The roots of the theory of word-representable graphs are in the study of the celebrated Perkins semigroup \cite{KS,Seif} which has played a central role in semigroup theory since 1960, particularly as a source of examples and counterexamples. However, the most interesting aspect of word-representable graphs from an algebraic point of view seems to be the notion of a {\em semi-transitive orientation} \cite{HKP}, which generalizes partial orders. It was shown in  \cite{HKP} that a graph is word-representable if and only if it admits a semi-transitive orientation (see Section~\ref{sec4} for a definition of a semi-transitive orientation).

More motivation points to study word-representable graphs include the fact exposed in \cite{KL} that these graphs generalize several important classes of graphs such as {\em circle graphs} \cite{Cer}, 3-{\em colourable graphs} \cite{BE} and {\em comparability graphs} \cite{Lov}. Relevance of word-representable graphs to scheduling problems was explained in \cite{HKP} and it was based on \cite{GZ}. Furthermore, the study of word-representable graphs is interesting from an algorithmic point of view as explained in \cite{KL}. For example, the {\em Maximum Clique problem} is polynomially solvable on word-representable graphs \cite{KL} while this problem is generally NP-complete \cite{BBPP}. Finally, word-representable graphs are an important class among other graph classes considered in the literature that are defined using words. Examples of other such classes of graphs are {\em polygon-circle graphs} \cite{Koebe} and {\em word-digraphs} \cite{Bell}. 

In relation to the main focus in our paper, one can prove the following theorem. 

\begin{theorem}\label{two-numbers} Toeplitz graphs $G_n(z^s,z)$ and  $G_n(z^s+z^t,z)$, $s< t$, are word-representable for any $n\geq 1$. \end{theorem}

\begin{proof} It is easy to see by induction on the number of vertices that  $G_n(z^s+z^t,z)$, and thus $G_n(z^s,z)$, is 3-colorable. Indeed, the base case is trivial, and assuming $G_{n-1}(z^s+z^t,z)$ is  3-colorable, we can obtain  $G_n(z^s+z^t,z)$ by increasing the labels of its vertices by 1, and adding the new vertex labeled by 1 to it. Since the new vertex is connected to (no more than) two vertices, there is at least one colour of three colours
available for it. Since 3-colorable graphs are word-representable \cite{HKP}, we are done. \end{proof}

\subsection{Infinite word-representable graphs}

For an {\em infinite} graph $G=G(g,f)$ defined by the relation (\ref{inf-def}) we say that $G$ is word-representable if {\em each} finite graph $G_n(g,f)$ is word-representable.  We note that the notion of an {\em infinite word-representable graph} was never considered in the literature.

Define the {\em index of word-representability} IWR($G$) of an infinite graph $G=G(g,f)$ as the {\em largest} $n$ such that $G_n(g,f)$ is word-representable. Since any graph on at most five vertices is word-representable \cite{KL}, we have that IWR($G$)$\geq 5$ for any $G$. If $G$ is word-representable, we let IWR($G$)$=\infty$.

As corollaries to Theorems~\ref{two-numbers}, \ref{010} and \ref{thm:101} and Corollary~\ref{01,10} and \ref{001}, we obtain many examples of infinite word-representable graphs. In particular, it follows from Corollary~\ref{01,10} that the Fibonacci matrix defines a Toeplitz graph with the index of word-representability $\infty$. On the other hand, it can be checked using \cite{Glen} that the Pascal triangle and the Catalan triangle define Riordan graphs with the index of word-representability 11 and 12, respectively. The smallest non-word-representable Pascal and Catalan graphs are given, respectively, by the following adjacency matrices
$$\begin{array}{cc}
\left(\begin{array}{cccccccccccc}
0&1&1&1&1&1&1&1&1&1&1&1\\
1&0&1&0&1&0&1&0&1&0&1&0\\
1&1&0&1&1&0&0&1&1&0&0&1\\
1&0&1&0&1&0&0&0&1&0&0&0\\
1&1&1&1&0&1&1&1&1&0&0&0\\
1&0&0&0&1&0&1&0&1&0&0&0\\
1&1&0&0&1&1&0&1&1&0&0&0\\
1&0&1&0&1&0&1&0&1&0&0&0\\
1&1&1&1&1&1&1&1&0&1&1&1\\
1&0&0&0&0&0&0&0&1&0&1&0\\
1&1&0&0&0&0&0&0&1&1&0&1\\
1&0&1&0&0&0&0&0&1&0&1&0
\end{array} \right)
&
\left(\begin{array}{ccccccccccccc}
0&1&1&0&1&0&0&0&1&0&0&0&0\\
1&0&1&0&1&0&0&0&1&0&0&0&0\\
1&1&0&1&1&1&0&0&1&1&0&0&0\\
0&0&1&0&1&0&0&0&1&0&0&0&0\\
1&1&1&1&0&1&1&0&1&1&1&0&0\\
0&0&1&0&1&0&1&0&1&0&1&0&0\\
0&0&0&0&1&1&0&1&1&1&0&1&0\\
0&0&0&0&0&0&1&0&1&0&0&0&0\\
1&1&1&1&1&1&1&1&0&1&1&0&1\\
0&0&1&0&1&0&1&0&1&0&1&0&1\\
0&0&0&0&1&1&0&0&1&1&0&1&1\\
0&0&0&0&0&0&1&0&0&0&1&0&1\\
0&0&0&0&0&0&0&0&1&1&1&1&0
\end{array} \right)
\end{array}.$$
Also, from Section~\ref{sec3} we see that for $G(a_1 a_2 \cdots a_m)=\lim_{n\to\infty}G_n(a_1 a_2 \cdots a_m)$, $$\mbox{IWR}(G(10^21^5))=\mbox{IWR}(G(10^21^4))=\mbox{IWR}(G(101^4))=\mbox{IWR}(G(101^3))=8,$$ which are the smallest non-word-representable Riordan graphs. 

\subsection{Comparability graphs}
An orientation of a graph is {\em transitive}, if the presence of the edges $u\rightarrow v$ and $v\rightarrow z$ implies the presence of the edge $u\rightarrow z$. An oriented graph $G$ is a {\em comparability graph} if $G$ admits a transitive orientation. 
A graph is {\em permutationally representable} if it can be represented by concatenation of permutations of (all) vertices. Thus, permutationally representable graphs are a subclass of word-representable graphs. The following theorem classifies these graphs.

\begin{theorem}[\cite{KS}]\label{comp-thm} A graph is permutationally representable if and only if it is a comparability graph. \end{theorem} 


%

Note that $G$ is a comparability graph if its adjacency matrix $A=(a_{i,j})_{1\leq i,j\leq n}$ satisfies the following: any time when $a_{i,j}=1$ and $a_{j,k}=1$ for $i<j<k$ we also have $a_{i,k}=1$.
To see this, one can obtain a transitive orientation of $G$ by orienting each edge $ij$ as $i\rightarrow j$ whenever $i < j$.
In particular, we have the following statement by the transitivity of the congruence relation.
Observe that in the Toeplitz graph $G_n(0^{k-1} 1)$ a vertex $i$ is adjacent to a vertex $j$ if and only if $i \equiv j$ $(\text{mod}~k)$.
\begin{cor}\label{001}
For any positive integers $n$ and $k$, the Toeplitz graph $G_n(0^{k-1} 1)$ is permutationally representable. 
\end{cor}

\subsection{Our results in this paper}

We already stated some results on word-representability of Toeplitz graphs in Theorem~\ref{two-numbers} and Corollary~\ref{001}. For another result, we note that the Toeplitz graph defined by  $\left(\frac{1}{1-z},z\right)$ is word-representable, because $G_n\left(\frac{1}{1-z},z\right)$ is a complete graph on $n$ vertices and it can be represented by any permutation of length $n$.  For word-representation of Riordan graphs beyond Toeplitz graphs, we note that the graph defined by $\left(1+z+\cdots +z^{k-1},z^{k}\right)$, $k\geq 1$, is also word-representable because each of such graphs $G_n$ is clearly a tree, and any tree can be represented using two copies of each letter~\cite{KL}. (Word-representability of a tree also follows from Theorem~\ref{comp-thm} since any tree is a comparability graph.) The observation on $\left(1+z+\cdots +z^{k-1},z^{k}\right)$ can be generalized to Riordan graphs defined by $(P(z),z^{m})$ where $P(z)$ is any polynomial of degree $k-1$ and $m\geq k$. Indeed, any such Riordan graph is a forest, so that it is a comparability graph and is word-representable by Theorem~\ref{comp-thm}. 

In either case, the main results in this paper are establishing word-representability of 
\begin{itemize}
\item $G_n(0^k 1^{m-k-l}0^l)$ for any positive integers $k,l,m,n$ such that $k+l < m$ in Theorem~\ref{010}; and
\item $G_n(1^k 0^m)$ and $G_n(0^k 1^m)$ for any positive integers $k, m$ and $n$ in Corollary~\ref{01,10}, which generalizes Corollary~\ref{001}; and
\item $G_n(1^{k-1}01^{m-k})$ for any positive integers $k < m$ satisfying either $\gcd(k,m)=1$ or $k=\frac{m}{2}$ in Theorem~\ref{thm:101}.
\end{itemize}

We will also show in Theorem~\ref{thm:decomp} that word-representability of a graph $G_n(a_1a_2\cdots a_m)$ implies that for each positive divisor $d$ of $m$, $G_{\lfloor \frac{n}{d}\rfloor}(a_d a_{2d} \cdots a_m)$ is also word-representable. The latter gives a way to construct non-word-representable Toeplitz graphs.


\medskip

\subsection{Proofs in this paper}\label{proofs-in-paper} 

All of our general statements contain rigorous proofs, e.g. in terms of explicit words representing various graphs.
However, in many other situations we have to refer to the results obtained using the freely-available user-friendly software \cite{Glen} created by Marc Glen to keep the paper being of reasonable size. Each of such results can be verified by hand as follows. If we claim that a (small) graph is word-representable, then \cite{Glen} can produce a word-representant for the graph, which can be used as a certificate. On the other hand, if we claim that a (small) graph is non-word-representable (based on the results of \cite{Glen}), then this can be checked using the notion of a semi-transitive orientation (see \cite{CKL} or \cite[Section 4.5]{KL} for a detailed explanation, illustrated on a particular graph, of how to do such a check; also, see Section~\ref{sec4}). 

\section{Word-representable Toeplitz graphs}\label{sec2}
In this section, we investigate for which $a_1,\dots,a_m \in \{0,1\}$ the graph $G_n(a_1a_2\cdots a_m)$ is word-representable for every $n$. We support our claims by explicit constructions of word-representants.
Given a subset $S$ of $[n]$, we let $u(S)$ (resp., $d(S)$) denote the words obtained by arranging the elements of $S$ in the increasing (resp., decreasing) order.

\subsection{Word-representability of $G_n(0^k 1^{l}0^m)$}
In this section, we prove that for every positive integer $n$, the graph $G_n(0^k 1^{l}0^m)$ is word-representable for each non-negative integers $k,l$ and $m$ such that $k + l + m \ge 1$.
A direct consequence implies that for every positive integers $k,l,m$ and $n$ the graph of the form either $G_n(0^k 1^l)$ or $G_n(1^l 0^m)$ is word-representable.
Note that the special case of $l=1$ and $m=0$ follows from Corollary~\ref{001}.

\begin{theorem}\label{010}
For any non-negative integers $k,l$ and $m$ such that $k + l + m \geq 1$, the graph $G_n(0^{k}1^{l} 0^{m})$ is word-representable for any positive integer $n$.
\end{theorem}
\begin{proof}
If $l = 0$, then the graph $G_n(0^k 1^l 0^m) = G_n(0^{k+m})$ is empty, which is word-representable.
Thus we may assume that $l \neq 0$.
If $k=m=0$, then the graph $G_n(0^k 1^l 0^m) = G_n(1^l)$ is a complete graph, which is word-representable.
Hence we may further assume that at least one of the integers $k$ and $m$ is positive.

To construct the word that represents the graph $G_n(0^k 1^l 0^m)$, we partition the set $[n]$ into $B_1\cup\cdots\cup B_{k+l+m}$, where for each $i\in[k+l+m]$,
$$B_i:=\{a \in [n]\ |\ a \equiv i \pmod{k+l+m}\}.$$
We then define $w_1(t)$ by the $1$-uniform word (permutation)
$$w_1(t):=d(B_1)\dots d(B_{t-1})d(B_t \cup \cdots \cup B_{t+k})d(B_{t+k+1})\dots d(B_{k+l+m})$$
for $1 \le t \le l+m$, and by the $2$-uniform word
$$w_1(t):=d(B_1)\dots d(B_{t-1})d(B_t \cup \cdots \cup B_{k+l+m} \cup B_1 \cup \cdots \cup B_{t-l-m})d(B_{t-l-m+1})\dots d(B_{k+l+m})$$
for $l+m+1 \le t \le k+l+m$.
Similarly, we define $w_2(t)$ by the $2$-uniform word
$$w_2(t):=d(B_1)\dots d(B_{k+l+t})u(B_{k+l+t+1}\cup \cdots \cup B_{k+l+m} \cup B_1 \cup \cdots \cup B_t)d(B_{t+1})\dots d(B_{k+l+m})$$
for $1 \le t \le m-1$, and by the $1$-uniform word
$$w_2(t):=d(B_1)\dots d(B_{t-m})u(B_{t-m+1}\cup \cdots \cup B_t)d(B_{t+1})\dots d(B_{k+l+m})$$
for $m \le t \le k+l+m$.
For $m = 0$, we define $w_2(t)$ to be the empty word for any $t$.
Note that the empty word does not affect alternation of any pair of vertices, so we do not have to consider the case of $m=0$ separately.
Now we claim that the graph $G_n(0^{k}1^{l} 0^{m})$ can be represented by the word 
$$W:=d(B_1)d(B_2)\dots d(B_{k+l+m})w_1(1)\dots w_1(k+l+m) w_2(1) \dots w_2(k+l+m).$$
We refer the Reader to Example~\ref{ex1} right after the proof illustrating the construction of $W$ in the case of $k=2$, $l=3$, $m=1$ and $n=11$.

Before we go through the details, we will briefly describe our proof idea.
We start with the $1$-uniform word $d(B_1)d(B_2)\dots d(B_{k+l+m})$ which represents the complete graph of order $n$.
Let $x \in B_i$ and $y \in B_j$ be two vertices in the graph $G_n(0^k 1^l 0^m)$.
Without loss of generality, we may assume that $x < y$.
Note that the word $d(B_1)d(B_2)\dots d(B_{k+l+m})$ contains either $xy$ or $yx$ as an induced subword.

We will first show that if the vertices $x$ and $y$ are adjacent in the graph $G_n(0^k 1^l 0^m)$, then $x$ and $y$ alternate in each of the words $w_1(t)$ and $w_2(t)$, in the same order as in the word $d(B_1)d(B_2)\dots d(B_{k+l+m})$.
For instance, if $x$ and $y$ are adjacent in the graph $G_n(0^k 1^l 0^m)$ and $d(B_1)d(B_2)\dots d(B_{k+l+m})$ contains $xy$ as an induced subword, then each of $w_1(t)$ and $w_2(t)$ contains either $xy$ or $xyxy$ as an induced subword.

Then we will show that if the vertices $x$ and $y$ are not adjacent in the graph $G_n(0^k 1^l 0^m)$, then $x$ and $y$ do not alternate in the word $W$.
Here, the words $w_1(i)$ and $w_2(i)$ play an important role.
For instance, depending on the condition of $i$ and $j$, the order of $x$ and $y$ in the subword $d(B_i \cup \dots \cup B_{i+k})$ of $w_1(i)$ may be different from that  in the word $d(B_i)\dots d(B_{i+k})$, breaking alternation of $x$ and $y$ in the word $W$.

\bigskip
\noindent{\bf Case 1.} Let $j-i \equiv a \pmod {k+l+m}$ for some $a \in \{k+1,k+2,\dots,k+l\}$.
In this case, $x$ and $y$ are adjacent in $G_n(0^k1^l0^m)$, so $x$ and $y$ must alternate in $W$. We consider two subcases.

\smallskip
\noindent (1) First we assume that $i < j$.
Then $j-i = a$ and the word $d(B_1)\dots d(B_{k+l+m})$ contains $xy$ as an induced subword.
Thus it is sufficient to show that for each $t$, both $w_1(t)$ and $w_2(t)$ contain either $xy$ or $xyxy$ as an induced subword.

For $w_1(t)$, we claim that $w_1(t)$ contains $xy$ as an induced subword if $1 \le t \le l+m$ and contains $xyxy$ as an induced subword if $l+m+1 \le t \le k+l+m$.
For $1 \le t \le l+m$, $x$ and $y$ cannot appear in the word $d(B_t \cup \cdots \cup B_{t+k})$ at the same time because $a \geq k+1$.
Thus it is clear that the $1$-uniform word
$$w_1(t)=d(B_1)\dots d(B_{t-1})d(B_t \cup \cdots \cup B_{t+k})d(B_{t+k+1})\dots d(B_{k+l+m})$$
contains $xy$ as an induced subword.
For $l+m+1 \le t \le k+l+m$, if $x$ and $y$ do not appear in the word $d(B_t \cup \cdots \cup B_{k+l+m} \cup B_1 \cup \cdots \cup B_{t-l-m})$ at the same time, then it is obvious that the $2$-uniform word 
$$w_1(t)=d(B_1)\dots d(B_{t-1})d(B_t \cup \cdots \cup B_{k+l+m} \cup B_1 \cup \cdots \cup B_{t-l-m})d(B_{t-l-m+1})\dots d(B_{k+l+m})$$
contains $xyxy$ as an induced subword.
If the word $d(B_t \cup \cdots \cup B_{k+l+m} \cup B_1 \cup \cdots \cup B_{t-l-m})$ contains both $x$ and $y$, then we have $t-l-m+1 \leq t \le j \le k+l+m$ and $1 \le i \le t-l-m \leq t-1$.
This implies that the $2$-uniform word $w_1(t)$ contains $xyxy$ as an induced subword.
Therefore, for each $t\in[k+l+m]$, $w_1(t)$ contains either $xy$ or $xyxy$ as an induced subword.

For $w_2(t)$, we claim that $w_2(t)$ contains $xyxy$ as an induced subword if $1 \le t \le m-1$, and contains $xy$ as an induced subword if $m \le t \le k+l+m$.
For $1 \le t \le m-1$, it is obvious that the $2$-uniform word
$$w_2(t)=d(B_1)\dots d(B_{k+l+t})u(B_{k+l+t+1}\cup \cdots \cup B_{k+l+m} \cup B_1 \cup \cdots \cup B_t)d(B_{t+1})\dots d(B_{k+l+m})$$
contains $xyxy$ as an induced subword if $x$ and $y$ do not appear in the word $u(B_{k+l+t+1}\cup \cdots \cup B_{k+l+m} \cup B_1 \cup \cdots \cup B_t)$ at the same time.
On the other hand, since we have $j-i =a \le k+l$, we observe that it is impossible to have both $k+l+t+1 \le j \le k+l+m$ and $1 \le i \le t$.
Hence if the word $u(B_{k+l+t+1}\cup \cdots \cup B_{k+l+m} \cup B_1 \cup \cdots \cup B_t)$ contains both $x$ and $y$, then it must be that either $k+l+t+1 \le i<j \le k+l+m$ or $1\le i<j\le t$.
In any case, the $2$-uniform word $w_2(t)$ contains $xyxy$ as an induced subword.
For $m \le t \le k+l+m$, clearly the $1$-uniform word 
$$w_2(t)=d(B_1)\dots d(B_{t-m})u(B_{t-m+1}\cup \cdots \cup B_t)d(B_{t+1})\dots d(B_{k+l+m})$$
contains $xy$ as an induced subword since $x<y$.

\smallskip
\noindent (2) Now we assume that $i>j$.
In this case, $j-i = a-(k+l+m)$ and the word $d(B_1)\dots d(B_{k+l+m})$ contains $yx$ as an induced subword.
We will show that for each $t\in[k+l+m]$, $w_1(t)$ and $w_2(t)$ contain either $yx$ or $yxyx$ as an induced subword.

For $w_1(t)$, from the assumption $x < y$, it is clear that the $1$-uniform word
$$w_1(t)=d(B_1)\dots d(B_{t-1})d(B_t \cup \cdots \cup B_{t+k})d(B_{t+k+1})\dots d(B_{k+l+m})$$
contains $yx$ as an induced subword when $1 \le t \le l+m$.
Hence it remains to show that for $l+m+1 \le t \le k+l+m$, the $2$-uniform word
$$w_1(t)=d(B_1)\dots d(B_{t-1})d(B_t \cup \cdots \cup B_{k+l+m} \cup B_1 \cup \cdots \cup B_{t-l-m})d(B_{t-l-m+1})\dots d(B_{k+l+m})$$
contains $yxyx$ as an induced subword.
This claim is obviously true when $x$ and $y$ do not appear in the word $d(B_t \cup \cdots \cup B_{k+l+m} \cup B_1 \cup \cdots \cup B_{t-l-m})$ at the same time.
If the word $d(B_t \cup \cdots \cup B_{k+l+m} \cup B_1 \cup \cdots \cup B_{t-l-m})$ contains both $x$ and $y$, then it must be that either $t \le j<i \le k+l+m$ or $1 \le j <i \le t-l-m$.
For this, observe that it is impossible to have both $t \le i \le k+l+m$ and $1 \le j \le t-l-m$.
If not, we have $j-i=a-(k+l+m) \le -l-m$, which implies that $a \le k$.
This contradicts to the fact that $a \geq k+1$.
Thus it follows that the $2$-uniform word $w_1(t)$ contains $yxyx$ as an induced subword.

Finally, we claim that $w_2(t)$ contains $yxyx$ as an induced subword if $1 \le t \le m-1$, and contains $yx$ as an induced subword if $m \le t \le k+l+m$.
For $1\le t \le m-1$, as before, the statement is obviously true if $x$ and $y$ do not appear at the same time in the word $u(B_{k+l+t+1}\cup \cdots \cup B_{k+l+m} \cup B_1 \cup \cdots \cup B_t)$. 
Note that from $a \leq k+l$, we have that $i-j=(k+l+m)-a \ge m$.
Thus if the word $u(B_{k+l+t+1}\cup \cdots \cup B_{k+l+m} \cup B_1 \cup \cdots \cup B_t)$ contains both $x$ and $y$, then it must be that $t+1 < k+l+t+1 \le i \le k+l+m$ and $1 \le j \le t < k+l+t$. Thus the $2$-uniform word
$$w_2(t)=d(B_1)\dots d(B_{k+l+t})u(B_{k+l+t+1}\cup \cdots \cup B_{k+l+m} \cup B_1 \cup \cdots \cup B_t)d(B_{t+1})\dots d(B_{k+l+m})$$
contains $yxyx$ as an induced subword.
For $m \le t \le k+l+m$, since $i-j \ge m$, the word $u(B_{t-m+1}\cup \cdots \cup B_t)$ cannot contain both $x$ and $y$, so that the $1$-uniform word
$$w_2(t)=d(B_1)\dots d(B_{t-m})u(B_{t-m+1}\cup \cdots \cup B_t)d(B_{t+1})\dots d(B_{k+l+m})$$
contains $yx$ as an induced subword.

\bigskip
\noindent{\bf Case 2.} Let $j-i \equiv a \pmod {k+l+m}$ for some $a \in \{1,\dots,k\}$.
In this case, $x$ and $y$ are not adjacent in $G_n(0^k1^l0^m)$, so $x$ and $y$ must not alternate in $W$. We consider two subcases.

\smallskip
\noindent (1) First assume that $i<j$.
Then $j = i+a$ and the word $d(B_1)\dots d(B_{k+l+m})$ contains $xy$ as an induced subword.
If $i \le l+m$, then we have $j = i+a \leq i+k$, and hence the word $d(B_i \cup \cdots \cup B_{i+k})$ contains both $x$ and $y$.
Thus the $1$-uniform word
$$w_1(i)=d(B_1)\dots d(B_{i-1})d(B_i \cup \cdots \cup B_{i+k})d(B_{i+k+1})\dots d(B_{k+l+m})$$
contains $yx$ as an induced subword, which implies that $x$ and $y$ do not alternate in $W$.
If $1+l+m \le i \le k+l+m$, then since $i < j \le k+l+m$, the word $d(B_i \cup \cdots \cup B_{k+l+m} \cup B_1 \cup \cdots \cup B_{i-l-m})$ contains exactly one $x$ and exactly one $y$.
Thus the $2$-uniform word 
$$w_1(i)=d(B_1)\dots d(B_{i-1})d(B_i \cup \cdots \cup B_{k+l+m} \cup B_1 \cup \cdots \cup B_{i-l-m})d(B_{i-l-m+1})\dots d(B_{k+l+m})$$
contains $yxxy$ as an induced subword, so $x$ and $y$ do not alternate in $W$.

\smallskip
\noindent (2) Next we assume that $i>j$.
Then $j = i + a - (k+l+m)$ and the word $d(B_1)\dots d(B_{k+l+m})$ contains $yx$ as an induced subword.
Since $j=i+a-(k+l+m) \le i-l-m$, we have $i \ge l+m+1$.
Thus the word $d(B_i \cup \cdots \cup B_{k+l+m} \cup B_1 \cup \cdots \cup B_{i-l-m})$ contains exactly one $x$ and exactly one $y$. Therefore the word $$w_1(i)=d(B_1)\dots d(B_{i-1})d(B_i \cup \cdots \cup B_{k+l+m} \cup B_1 \cup \cdots \cup B_{i-l-m})d(B_{i-l-m+1})\dots d(B_{k+l+m})$$ contains $yyxx$ as an induced subword, so $x$ and $y$ do not alternate in $W$.

\bigskip
\noindent{\bf Case 3.}
Let $j-i \equiv a \pmod {k+l+m}$ for some $a \in \{k+l+1,\dots,k+l+m\}$.
Again, $x$ and $y$ are not adjacent in $G_n(0^k1^l0^m)$ in this case, and hence $x$ and $y$ must not alternate in $W$.
The argument for this case is similar to that for {\bf Case 2}, but here we will use the word $w_2(i)$ instead of the word $w_1(i)$ in {\bf Case 2}.

\smallskip
\noindent (1) First assume that $i <j$.
Then $j = i+a$ and $d(B_1)\dots d(B_{k+l+m})$ contains $xy$ as an induced subword.
Since $j=i+a \ge i+k+l+1$, we have $i \le m-1$. Thus the word $u(B_{k+l+i+1}\cup \cdots \cup B_{k+l+m} \cup B_1 \cup \cdots \cup B_i)$ contains exactly one $x$ and exactly one $y$.
Thus the $2$-uniform word
$$w_2(i)=d(B_1)\dots d(B_{k+l+i})u(B_{k+l+i+1}\cup \cdots \cup B_{k+l+m} \cup B_1 \cup \cdots \cup B_i)d(B_{i+1})\dots d(B_{k+l+m})$$
contains $xxyy$ as an induced subword, and hence $x$ and $y$ do not alternate in $W$.

\smallskip
\noindent (2) Next we assume $i \ge j$.
Then $j = i+a-(k+l+m)$ and $d(B_1)\dots d(B_{k+l+m})$ contains $yx$ as an induced subword.
If $m \le i \le k+l+m$, then since $j=i+a -(k+l+m)\ge i-m+1$, the word $u(B_{i-m+1}\cup \cdots \cup B_i)$ contains exactly one $x$ and exactly one $y$.
Thus the $1$-uniform word
$$w_2(i)=d(B_1)\dots d(B_{i-m})u(B_{i-m+1}\cup \cdots \cup B_i)d(B_{i+1})\dots d(B_{k+l+m})$$
contains $xy$ as an induced subword.
If $1\le i \le m-1$, then the word $u(B_{k+l+i+1}\cup \cdots \cup B_{k+l+m} \cup B_1 \cup \cdots \cup B_i)$ contains exactly one $x$ and exactly one $y$ since $1 \le j \le i \le m-1$.
Thus the word $$w_2(i)=d(B_1)\dots d(B_{i-m})u(B_{i-m+1}\cup \cdots \cup B_i)d(B_{i+1})\dots d(B_{k+l+m})$$ contains $yxxy$ as an induced subword.
In either case, $x$ and $y$ do not alternate in $W$.
\end{proof}

\begin{example}\label{ex1} {\em We illustrate the construction of the word $W$ in Theorem~\ref{010} in the case of $k=2$, $l=3$, $m=1$ and $n=11$. In this case, $B_1=\{1,7\}$, $B_2=\{2,8\}$, $B_3=\{3,9\}$, $B_4=\{4,10\}$, $B_5=\{5,11\}$, and $B_6=\{6\}$. 
Thus, $W$ is obtained by concatenating  the following words
$$\begin{array}{l}
d(B_1)d(B_2)\dots d(B_{6})=71~82~93~(10)4~(11)5~6 \\
w_1(1) = d(B_1\cup B_2 \cup B_3) d(B_4) d(B_5) d(B_6) = 987321~(10)4~(11)5~6 \\
w_1(2) = 71~(10)98432~(11)5~6 \\
w_1(3) = 71~82~(11)(10)9543~6 \\
w_1(4) = 71~82~93~(11)(10)654 \\
w_1(5) = 71~82~93~(10)4~(11)7651~82~93~(10)4~(11)5~6\\
w_1(6) = 71~82~93~(10)4~(11)5~87621~93~(10)4~(11)5~6\\
w_2(1) = u(B_1) d(B_2) d(B_3) d(B_4) d(B_5) d(B_6) = 17~82~93~(10)4~(11)5~6\\
w_2(2) = 71~28~93~(10)4~(11)5~6\\
w_2(3) = 71~82~39~(10)4~(11)5~6\\
w_2(4) = 71~82~93~4(10)~(11)5~6\\
w_2(5) = 71~82~93~(10)4~5(11)~6\\
w_2(6) = 71~82~93~(10)4~(11)5~6.\end{array}$$
It can be checked using \cite{Glen} that this $W$ indeed represents $G(0^21^30)$.}
\end{example}

Observe that in Theorem~\ref{010}, we do not have to consider {\bf Case 2} if $k=0$, and we do not have to consider {\bf Case 3} if $m=0$.
This allows us to provide shorter words that represent the graphs $G_n(0^k 1^l)$ and $G_n(1^l 0^m)$, respectively.
They will be described in the following corollary.
\begin{cor}\label{01,10}
For any positive integers $k,l,m$ and $n$, the graphs $G_n(0^{k}1^{l})$ and $G_n(1^{l}0^{m})$ are word-representable.
\end{cor}
\begin{proof}
Let $k,l$ and $n$ be fixed.
Then the graph $G_n(0^k 1^l)$ is precisely the case when $k$ and $l$ are positive and $m=0$ in Theorem~\ref{010}.
By the above observation, {\bf Case 1} and {\bf Case 2} in Theorem~\ref{010} imply that the graph $G_n(0^{k}1^{l})$ can be represented by the word
$$W':=d(B_1)\dots d(B_{k+l})w_1(1)\dots w_1(k+l).$$

Now let $l, m$ and $n$ be fixed.
Then the graph $G_n(1^l 0^m)$ is precisely the case when $l$ and $m$ are positive and $k=0$ in Theorem~\ref{010}.
By the above observation, {\bf Case 1} and {\bf Case 3} in Theorem~\ref{010} imply that the graph $G_n(1^l 0^m)$ can be represented by the word 
$$W'':=d(B_1)\dots d(B_{l+m})w_2(1) \dots w_2(l+m).$$

%
%
\end{proof}

\subsection{Word-representability of $G_n(1^{k-1}01^{m-k})$}
Here we deal with word-representability of $G_n(1^{k-1}01^{m-k})$ under given assumptions.
The condition that either $\gcd(k,m)=1$ or $k=\frac{m}{2}$ cannot be removed because of the existence of non-word-representable graphs examined by \cite{Glen}.
When we have $m=6$ and $k=2$, the graph $G_9(101111)$ is not word-representable.
Note that in this case we have ${\gcd}(k,m)=2$.
Moreover, the assertion of Theorem~\ref{thm:101} is not naturally extended to arbitrary graphs of type $G_n(1^k 0^{m-k-l} 1^l)$.
For instance, $G_{12}(1110011)$ is not word-representable.

\begin{theorem}\label{thm:101}
If positive integers $k,m$ such that $k < m$ satisfy ${\rm gcd}(k,m)=1$ or $k=\frac{m}{2}$ (when $m$ is even), $G_n(1^{k-1}01^{m-k})$ is word-representable for any $n$.
\end{theorem}
\begin{proof}
Let $[n]=B_1\cup\cdots\cup B_{m}$ be a partition of $[n]$ such that for each $i\in[m]$, $B_i:=\{a \in [n]\ |\ a \equiv i \pmod{m}\}$.

\bigskip
\noindent{\bf Case 1.} We first consider when ${\rm gcd}(k,m)=1$.
Note that $B_1,B_{k+1},B_{2k+1},\dots,B_{(m-1)k+1}$ are all distinct.
If not, then it must be that $B_{ak+1}=B_{bk+1}$ for some $0 \le a <b \le m-1$, implying that $(b-a)k$ can be divided by $m$. Since ${\rm gcd}(k,m)=1$, we can conclude that $b-a$ can be divided by $m$, which is not true because $0<b-a<m$.

Now for each $1 \le t \le m-1$, we define a $1$-uniform word $w_t$ by
$$w_t:=d(B_1)\dots d(B_{(t-2)k+1})d(B_{(t-1)k+1} \cup B_{tk+1})d(B_{(t+1)k+1})\dots d(B_{(m-1)k+1}),$$
and we define a $2$-uniform word $w_m$ by
$$w_m:=d(B_1)\dots d(B_{(m-2)k+1})d(B_{(m-1)k+1}\cup B_1)d(B_{k+1})\dots d(B_{(m-1)k+1}).$$
We claim that the word $W:=w_1 w_2\dots w_m$ represents the graph $G_n(1^{k-1}01^{m-k})$. Suppose $x \in B_{ik+1}$ and $y \in B_{jk+1}$ for some $i,j \in \{0,1,\dots,m-1\}$. Without loss of generality, we may assume that $x < y$.

\noindent (1) Assume that either $j=i+1$ or $i=m-1$ and $j=0$.
Then $x$ and $y$ are not adjacent in $G_n(1^{k-1}01^{m-k})$, thus $x$ and $y$ should not alternate in $W$. 
If $j=i+1$, then the word $w_i$ contains $xy$ as an induced subword, and the word $w_j$ contains $yx$ as an induced subword.
If we have $i=m-1$ and $j=0$, then the word $w_m$ contains $yyxx$ as an induced subword.
In both cases, $x$ and $y$ do not alternate in the word $W$.

\smallskip

\noindent (2) Otherwise, $x$ and $y$ are adjacent in $G_n(1^{k-1}01^{m-k})$, and hence $x$ and $y$ should alternate in $W$.
If $i<j$, then for each $1 \le t \le m-1$, the word $w_t$ contains $xy$ as an induced subword since $j \ne i+1$, and the word $w_m$ contains $xyxy$ as an induced subword.
If $i \ge j$, then for each $1 \le t \le m-1$, the word $w_t$ contains $yx$ as an induced subword, and the word $w_m$ contains $yxyx$ as an induced subword since $x$ and $y$ do not appear in the word $d(B_{(m-1)k+1} \cup B_1)$ at the same time.
In both cases, $x$ and $y$ alternate in $W$.

\bigskip
\noindent{\bf Case 2.} Now we consider when $m$ is even and $k=\frac{m}{2}$.
In this case, the $2$-uniform word 
$$W:=d(B_1)d(B_{k+1})\dots d(B_k)d(B_{2k})d(B_{k+1})d(B_1)\dots d(B_{2k})d(B_k)$$
represents the graph $G_n(1^{k-1}01^{m-k})$.
To see this, suppose $x \in B_i$ and $y \in B_j$ for some $i,j \in [m]$, and assume that $x < y$.

If $j=i\pm k$, then $x$ and $y$ are not adjacent in $G_n(1^{k-1}01^{m-k})$.
Observe that the word $W$ contains $xyyx$ as an induced subword if $j=i+k$, and that $W$ contains $yxxy$ as an induced subword if $j=i-k$.
Thus $x$ and $y$ do not alternate in $W$.

Otherwise, i.e. if $j\neq i\pm k$, $x$ and $y$ are adjacent in $G_n(1^{k-1}01^{m-k})$.
Clearly, the word $W$ contains $xyxy$ as an induced subword if $i' < j'$ and $W$ contains $yxyx$ as an induced subword if $i' > j'$, where $i',j'\in[k]$ and $i \equiv i'$ $(\text{mod}~k)$ and $j \equiv j'$ $(\text{mod}~k)$.
Thus $x$ and $y$ alternate in $W$.
\end{proof}

\section{Non-word-representable Toeplitz graphs}\label{sec3}
As is mentioned above, not all Toeplitz graphs are word-representable.  Using \cite{Glen} we see that the smallest non-word-representable Toeplitz graph has nine vertices.
An example of such a graph is  $G_9(10^21^5)$ given by the adjacency matrix 
$$
\left(
\begin{array}{ccccccccc}
0 &1 & 0 & 0 & 1 & 1 & 1 & 1 & 1\\
1 & 0& 1 & 0 & 0 & 1 & 1 &1& 1\\
 0 & 1 & 0 & 1 & 0 & 0 & 1 & 1 & 1 \\
 0 & 0 & 1 & 0 & 1 & 0 & 0 & 1 & 1\\
 1 & 0 & 0 & 1& 0& 1 & 0 & 0 & 1 \\
 1 & 1 & 0 & 0 & 1 & 0 & 1 & 0 & 0\\
 1 & 1 & 1 & 0 & 0 & 1 & 0 & 1 & 0\\
 1 & 1 & 1 & 1 & 0 & 0 & 1 & 0& 1\\
 1 & 1 & 1 & 1 & 1 & 0 & 0 &1 & 0
\end{array}
\right).
$$

%

In this section, we give a necessary condition on the word-representability of a Toeplitz graph.
To do this, we first prove that the induced subgraph of a Toeplitz graph $G_n(a_1\cdots a_m)$ on the vertex subset $\{di:1\leq i\leq \lfloor\frac{n}{d}\rfloor\}$ is also a Toeplitz graph.
The specific case when $d = 2$ implies Theorem~3.12 (iii) in \cite{CJKM}.
\begin{lemma}\label{lem:decomp}
Let $A$ be an $n \times n$ Riordan matrix defined by $\left(f_1+f_2z+\cdots+f_{n-1}z^{n-2},z\right)$. 
For a positive integer $d$, let $k=\lfloor \frac{n}{d}\rfloor$ and $A_d$ be the submatrix of $A$ induced by columns $d,2d,\dots, kd$ in $A$. 
Then the matrix $A_d$ is a $k \times k$ Riordan matrix defined by $\left(f_d+f_{2d}z+\cdots+f_{(k-1)d}z^{k-2},z\right)$.
\end{lemma}
\begin{proof}
For each $i,j \in [k]$ such that $i<j$, the $(i,j)$ entry of $A_d$ is equal to the $(id,jd)$ entry of $A$, which is equal to $f_{(j-i)d}$. 
Thus $A_d$ is the $k \times k$ Riordan matrix which is defined by 
$\left(f_d+f_{2d}z+\cdots+f_{(k-1)d}z^{k-2},z\right)$.
\end{proof}
\begin{theorem}\label{thm:decomp}
Let $G_n(a_1a_2\cdots a_m)$ be a word-representable Toeplitz graph. Then, for each positive divisor $d$ of $m$, $G_{\lfloor \frac{n}{d}\rfloor}(a_d a_{2d} \cdots a_m)$ is word-representable.
\end{theorem}
\begin{proof}
Let $A$ be the adjacency matrix of $G_n(a_1a_2\cdots a_m)$. Then $A$ is the $n \times n$ Riordan matrix defined by 
$$\left( \frac{a_1+a_2 z+\cdots+a_m z^{m-1}}{1-z^m},z \right).$$
In other words, $A$ is the $n \times n$ Riordan matrix defined by 
$\left(f_1+f_2z+\cdots+f_{n-1}z^{n-2},z\right)$ where $f_i$ is given by $f_i = a_{i-\lfloor\frac{i-1}{m}\rfloor m}$ for each $i\in[n-1]$.
Then, by Lemma~\ref{lem:decomp} we obtain that the submatrix $A_d$ of $A$ is the Riordan matrix defined by 
$$\left(f_d+f_{2d}z+\cdots+f_{(k-1)d}z^{k-2},z\right)$$ for $k=\lfloor \frac{n}{d}\rfloor$. 
Thus $A_d$ is the adjacency matrix of the Toeplitz graph with $\lfloor \frac{n}{d}\rfloor$ vertices defined by
$$\left( \frac{a_d+a_{2d} z+\cdots+a_{m} z^{m-1}}{1-z^{m}},z \right).$$
Therefore, $G_{\lfloor \frac{n}{d}\rfloor}(a_d a_{2d} \cdots a_m)$ is an induced subgraph of $G_n(a_1a_2\cdots a_m)$, which implies that $G_{\lfloor \frac{n}{d}\rfloor}(a_d a_{2d} \cdots a_m)$ is also word-representable by the heredity of the word-representablity.
\end{proof}

Theorem~\ref{thm:decomp} says that if a Toeplitz graph $G_{\lfloor \frac{n}{d} \rfloor}(a_d a_{2d}\cdots a_m)$ is not word-representable for some divisor $d$ of $m$ then the graph $G_n(a_1\cdots a_m)$ is not word-representable. This gives a way to construct non-word-representable Toeplitz graphs. For example, $G_{18}(010001010101)$ is not word-representable because $G_9(101111)$ is not word-representable. More generally, $G_{9k}(w_1 1 w_2 0 w_3 1 w_4 1 w_5 1 w_6 1 w_7 1 w_8 0 w_9 1)$ is not word-representable where $w_i$'s are words over $\{0,1\}$ of length $k-1$ since $G_9(101111)=G_9(101111101)$, so that
Theorem~\ref{thm:decomp} guarantees that among $2^{9k}$ Toeplitz graphs on $9k$ vertices, there are at least $2^{9(k-1)}$  non-word-representable graphs.

\section{Concluding remarks}\label{sec4}

In this paper we give several general classes of word-representable Toeplitz graphs, using explicit representation as a key approach. We were not able to apply similar approach to the other classes of Toeplitz graphs. In any case, further advances in the area, hopefully leading to a complete classification of word-representable Toeplitz graphs, or more generally Riordan graphs, may require usage of other tools, such as the powerful notion of a {\em semi-transitive orientation}. This notion has been used successfully in many situations (see \cite{K} for an overview), because it allows to bypass dealing with complicated constructions on words, and we complete this paper with describing it.   

The notion of a semi-transitive orientation was introduced in \cite{HKP}, but we follow \cite[Section 4.1]{KL} to introduce it here.
A graph $G=(V,E)$ is {\em semi-transitive} if it admits
an {\em acyclic} orientation such that for any directed path 
$v_1\rightarrow v_2\rightarrow \cdots \rightarrow v_k$ with $v_i\in V$ for all $i$, $1\leq i\leq k$, either
\begin{itemize}
\item there is no edge $v_1\rightarrow v_k$, or 
\item the edge $v_1\rightarrow v_k$ is present and there are edges $v_i\rightarrow v_j$ for all $1\le i<j\le k$. 
In other words, in this case, the (acyclic) subgraph induced by the vertices $v_1,\ldots,v_k$ is transitive (with the unique source $v_1$ and the unique sink $v_k$).  
\end{itemize}
We call such an orientation {\em semi-transitive}. In fact, the notion of a semi-transitive orientation is defined in \cite{HKP} in terms of {\em shortcuts} as follows. A {\em semi-cycle} is the directed acyclic
graph obtained by reversing the direction of one edge of a directed cycle in which the directions form a directed path. An acyclic digraph is a shortcut if it is induced by
the vertices of a semi-cycle and contains a pair of non-adjacent
vertices. Thus, a digraph on the vertex set $\{ v_1, \ldots,
v_k\}$ is a shortcut if it contains a directed path $v_1\rightarrow v_2\rightarrow \cdots
\rightarrow v_k$, the edge $v_1\rightarrow v_k$, and it is missing an edge $v_i\rightarrow v_j$ for some $1 \le i
< j \le k$; in particular, we must have $k\geq 4$, so that any shortcut is on at least four vertices. Clearly, this definition is just another way to introduce the notion of a semi-transitive orientation presented above. 

It is not difficult to see that
all transitive (that is, comparability) graphs are semi-transitive,
and thus semi-transitive orientations are a generalization of transitive orientations.  A key theorem in the theory of word-representable graphs is presented next, and we expect it to be of great use in the study of word-representable Riordan graphs.

\begin{theorem}[\cite{HKP}]\label{key-thm} A graph $G$ is word-representable if and only if it admits a semi-transitive orientation (that is, if and only if $G$ is semi-transitive).
\end{theorem}

\section*{Acknowledgments} This work was supported by the National Research Foundation of Korea (NRF) grant funded by the Korea government (MSIP) (2016R1A5A1008055) and the Ministry of Education of Korea (NRF-2016R1A6A3A11930452).


%

\end{document}